\documentclass[11pt]{amsart}
\usepackage[margin=1in]{geometry}
\usepackage{amsmath}
\usepackage{amssymb}
\usepackage{amsthm}
\usepackage{mathrsfs}
\usepackage{enumerate}
\usepackage{tikz}
\usepackage{hyperref}
\usepackage[indent]{parskip}
\newtheorem{thm}{Theorem}[section]
\newtheorem{lem}[thm]{Lemma}
\newtheorem{cor}[thm]{Corollary}
\theoremstyle{definition}
\newtheorem{dfn}[thm]{Definition}

\newtheorem*{prb*}{Problem}

\providecommand{\al}{\alpha}

\providecommand{\de}{\delta}
\providecommand{\De}{\Delta}
\providecommand{\e}{\varepsilon}
\providecommand{\f}{\varphi}
\providecommand{\sig}{\sigma}

\providecommand{\La}{\Lambda}

\providecommand{\emp}{\varnothing}

\providecommand{\F}{\mathbb{F}}
\providecommand{\N}{\mathbb{N}}

\providecommand{\R}{\mathbb{R}}

\providecommand{\cG}{\mathcal{G}}
\providecommand{\cH}{\mathcal{H}}

\providecommand{\cM}{\mathcal{M}}
\providecommand{\cP}{\mathcal{P}}

\DeclareMathOperator{\extr}{ex}

\numberwithin{equation}{section}

\title{Extremal graph theory and point configurations in Ahlfors-David regular sets}
\author{Alex McDonald}
\address{Mathematics Department, Kennesaw State University, Marietta, GA}
\email{amcdon79@kennesaw.edu}

\begin{document}
\begin{abstract}
We study the problem of embedding bipartite graphs in Ahlfors-David regular sets of large dimension using results from extremal graph theory.  Our main theorem states that any graph satisfying a power-improving bound on the extremal number can be found in the distance graph of a sufficiently high-dimensional AD-regular set.  In particular, we show that AD-regular sets of dimension greater than $\frac{d+1}{2}$ must contain even cycles of all lengths if $d\geq 3$, and must contain even cycles of length at least 6 if $d=2$.  This improves the best known threshold for the problem in $d\geq 4$, and yields entirely new results in $d=2,3$, under the extra assumption of AD-regularity.  We also prove analogous results for large subsets of vector spaces over finite fields, which improve the best known exponent for even cycles in all dimensions.
\end{abstract}
\maketitle

\section{Introduction}
An important class of problems of interest in harmonic analysis and geometric measure theory is to determine whether certain finite point configurations must occur in sets of large Hausdorff dimension, and if so, how abundant they must be.  The prototypical example of such a problem is the well-known Falconer distance problem, which asks how large the Hausdorff dimension of a compact set $E\subset \R^d$ ($d\geq 2$) must be to ensure the distance set
\[
\De(E):=\{|x-y|:x,y\in E\}
\]
has positive Lebesgue measure.  This problem was introduced by Falconer \cite{Falconer}, who proved dimension $\frac{d+1}{2}$ is sufficient and conjectured that $\frac{d}{2}$ is the optimal dimensional threshold.  Currently, the best results state that dimension greater than $s_d$ is sufficient, where
\[
s_d=
\begin{cases}
5/4, & d=2, \\[.1in]
\frac{d}{2}+\frac{1}{4}-\frac{1}{8d+4}, & d\geq 3.
\end{cases}
\]
The result in the plane is due to Guth, Iosevich, Ou, and Wang \cite{GIOW}, and the result in higher dimensions is due to Du, Ou, Ren, and Zhang \cite{DORZ}.  

In another direction, one can ask whether large Hausdorff dimension also implies that the distance set contains a non-degenerate interval.  The first result in this direction is due to Mattila and Sj\"{o}lin \cite{MS99}, who proved that $\dim(E)>\frac{d+1}{2}$ implies that $\De(E)$ contains an interval.  This matches the dimensional threshold in Falconer's original paper on the problem, but strengthens the conclusion.  Unlike in the positive-measure version of the problem, there has been no improvement in this threshold.

The problem we will consider in this paper is a generalization of the Falconer distance problem, where instead of trying to improve the threshold, one investigates more complicated patterns than single distances.  It will be useful to express our generalization in the language of graph theory.

\begin{dfn}
\label{distancegraphdefinition}
Let $G$ be a finite graph, let $E\subset \R^d$, and let $t>0$.  The \textbf{$t$-distance graph} of $E$ is the graph $\cG_t(E)$ with vertex set $E$ and adjacency defined by $x\sim y$ if and only if $|x-y|=t$.  The \textbf{$G$-distance set} of $E$ is the set
\[
\De_G(E)=\{t: \cG_t(E) \text{ contains a subgraph isomorphic to } G\}.
\]
\end{dfn}

Thus, $\De(E)=\De_{K_2}(E)\cup \{0\}$; that is, the distance set corresponds to the choice of the complete graph on two vertices, except that by convention we allow $0$ to be in the distance set but not the graph-distance set.  In general, a number $t\in \De_G(E)$ is a distance such that the vertex set of the graph $G$ can be embedded in the plane in such a way that every pair of adjacent vertices is mapped to a pair of points of distance $t$.

There are a number of results on the existence and abundance of distance graphs in sets of large Hausdorff dimension, both positive and negative.  In the positive direction, if $G=T$ is a tree, then Iosevich and Taylor \cite{IT19} proved that $\De_T(E)$ contains an interval whenever $E\subset \R^d$ is a compact set with $\dim E>\frac{d+1}{2}$, matching the Falconer/Mattila-Sj\"{o}lin threshold and generalizing the conclusion further.  Since a tree is simply a connected acyclic graph, the next natural question to ask is whether a similar result holds for a cycle.  We will use the following notation.

\begin{dfn}
For any $n\geq 3$, we denote by $C_n$ the graph consisting of $n$ vertices in a cycle; that is, the vertex set of $C_n$ is $\{1,\dots,n\}$, and the adjacency relation is defined by $1\sim 2\sim \cdots \sim n\sim 1$, as shown in Figure \ref{cycles}.
\end{dfn}

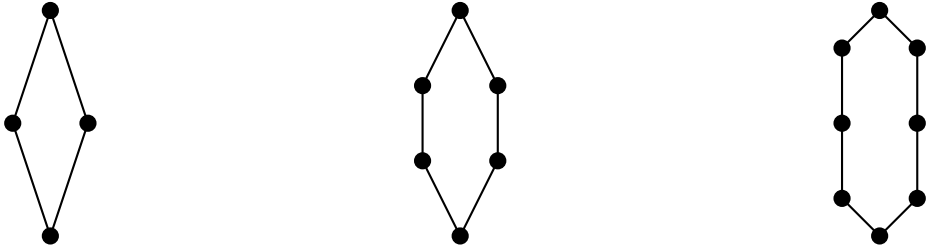
\begin{figure}
\centering
\begin{minipage}{.33\textwidth}
  \centering
\begin{tikzpicture}[scale=1, thick]
  \draw (.5,0)--(1,1.5)--(.5,3)--(0,1.5)--(.5,0);
  \draw [fill] (.5,0) circle [radius=0.1];
  \draw [fill] (1,1.5) circle [radius=0.1];
  \draw [fill] (.5,3) circle [radius=0.1];
  \draw [fill] (0,1.5) circle [radius=0.1];
\end{tikzpicture}
\end{minipage}%
\begin{minipage}{.33\textwidth}
  \centering
\begin{tikzpicture}[scale=1, thick]
 \draw (.5,0)--(1,1)--(1,2)--(.5,3)--(0,2)--(0,1)--(.5,0);
 \draw [fill] (.5,0) circle [radius=0.1];
  \draw [fill] (1,1) circle [radius=0.1];
  \draw [fill] (1,2) circle [radius=0.1];
  \draw [fill] (0,2) circle [radius=0.1];
  \draw [fill] (.5,3) circle [radius=0.1];
  \draw [fill] (0,1) circle [radius=0.1];
\end{tikzpicture}
\end{minipage}
\begin{minipage}{.33\textwidth}
  \centering
\begin{tikzpicture}[scale=1, thick]
  \draw (.5,0)--(1,.5)--(1,2.5)--(.5,3)--(0,2.5)--(0,.5)--(.5,0);
   \draw [fill] (.5,0) circle [radius=0.1];
  \draw [fill] (1,.5) circle [radius=0.1];
  \draw [fill] (1,2.5) circle [radius=0.1];
  \draw [fill] (1,1.5) circle [radius=0.1];
  \draw [fill] (0,2.5) circle [radius=0.1];
  \draw [fill] (0,1.5) circle [radius=0.1];
  \draw [fill] (.5,3) circle [radius=0.1];
  \draw [fill] (0,.5) circle [radius=0.1];
\end{tikzpicture}
\end{minipage}
\caption{The graphs $C_4, C_6, C_8$}
  \label{cycles}
\end{figure}

Greenleaf, Iosevich, and Pramanik \cite{GIP17} proved that if $d\geq 4$ and $E\subset \R^d$ is a compact set with $\dim E>\frac{d+3}{2}$, then for any $k\geq 2$ the set $\De_{C_{2k}}(E)$ contains an interval.  In the case $d=3$, the author, joint with Iosevich, Magyar, and B. McDonald \cite{IMMM25}, proved that if $E\subset \R^3$ satisfies $\dim E> 2.89$ then $\De_{C_4}(E)$ contains an interval.  In the negative direction, a result of Maga \cite{Maga} shows that there exists a set $E\subset \R^2$ such that $\dim E=2$, but such that $\De_{C_3}(E)=\De_{C_4}(E)=\emp$.

In this paper, we will assume not only that the Hausdorff dimension is large, but also that the set satisfies the Ahlfors-David regularity condition, which we define now.

\begin{dfn}
A set $E\subset \R^d$ is \textbf{Ahlfors-David regular} with exponent $s\in (0,d]$, or simply $s$-regular, if there exists a Borel probability measure $\mu$ supported on $E$ such that for every $x\in E$ and $0<r\leq \text{diam}(E)$, we have
\[
\mu(B_r(x))\approx r^s,
\]
where $B_r(x)$ is the ball of radius $r$ centered at the point $x$.  The measure $\mu$ is called an Ahlfors-David regular measure with exponent $s$, or simply an $s$-regular measure.
\end{dfn}
Here and throughout, we will use the notation $X\lesssim Y$ to mean that there exists a constant $C$ for which $X\leq CY$.  Similarly, $X\gtrsim Y$ means $Y\lesssim X$, and $X\approx Y$ means that both $X\lesssim Y$ and $X\gtrsim Y$ hold.  Expressions such as $X\lesssim_{a,b,c} Y$ mean that the implicit constant is allowed to depend on the parameters $a,b,c$.

Any $s$-regular set has Hausdorff dimension $s$, so this parameter may be compared to the dimensional thresholds in the previous results.  Our technique will show that any $s$-regular subset of $\R^d$ with $s>\frac{d+1}{2}$ can be discretized in such a way that a corresponding approximate distance graph has many edges.  This means that we will be able to exhibit approximate versions of a graph $G$ whenever we can show that any finite graph with enough edges must have a copy of $G$.  This leads us to the area of extremal graph theory, and the following definition.

\begin{dfn}
Let $G$ be a finite graph, and let $n\in\N$.  The \textbf{extremal number} $\extr(n,G)$ is the largest number of edges in any graph $H$ on $n$ vertices which does not have a subgraph isomorphic to $G$.
\end{dfn}

Our first main theorem is as follows.

\begin{thm}
\label{ADmain}
Let $E\subset \R^d$ be a compact, Ahlfors-David regular set with exponent $s$, and let $G$ be a graph satisfying
\[
\extr(n,G)\lesssim_{G,\al} n^{2-\al}
\]
for some $\al>0$.  If $s>\max(\frac{d+1}{2},\frac{1}{\al})$, then $\De_G(E)$ has non-empty interior.
\end{thm}

For simplicity, we will write $\lesssim$ in place of $\lesssim_{G,\al}$ in bounds such as the one in the hypothesis of Theorem \ref{ADmain}, with the understanding that the graph $G$ and power improvement $\al$ are to be understood as constants.  A classical result in graph theory says that $\extr(n,G)\approx n^2$ for all graphs $G$ which are not bipartite, so Theorem \ref{ADmain} is only non-trivial for bipartite graphs.  In this case, extimating extremal numbers is a highly active research area.  A result of Alon, Krivelevich, and Sudakov \cite{AKS03} (stated as Theorem \ref{AKStheorem} in this paper) gives the following corollary.

\begin{cor}
\label{aksADcorollary}
Let $E\subset \R^d$ be a compact, Ahlfors-David regular set with exponent $s$, let $G$ be a bipartite graph with parts\footnote{When we say $X$ and $Y$ are \textit{parts} of a bipartite graph, we mean that they are independent sets of vertices which form a partition of the vertex set.  Note that the sets $X$ and $Y$ are not determined uniquely by the graph.} $X$ and $Y$, and define
\[
r=\max_{v\in X}\deg(v).
\]
If $s>\max(\frac{d+1}{2},r)$, then then $\De_G(E)$ has non-empty interior.
\end{cor}

We will give several examples in Section \ref{applications}.  Since we have already discussed cycles, we mention one corollary of our corollary.

\begin{cor}
\label{cyclesADcorollary}
Let $E\subset \R^d$ be a compact, Ahlfors-David regular set with exponent $s>\frac{d+1}{2}$. 
\begin{enumerate}[(a)]
\item If $d=2$, then $\De_{C_{2k}}(E)$ has non-empty interior for every $k\geq 3$.
\item If $d\geq 3$, then $\De_{C_{2k}}(E)$ has non-empty interior for every $k\geq 2$.
\end{enumerate}
\end{cor}

Thus, under the additional assumption of Ahlfors-David regularity, we improve the dimensional thresholds from \cite{GIP17, IMMM25} in the case $d\geq 3$, and obtain the first non-trivial result in the case $d=2$.  We also note that in each case, the range of cycles allowed is best possible among even cycles.  In particular, Maga's aforementioned construction shows that it is not possible for any dimensional threshold to guarantee $\De_{C_4}(E)\neq \emp$ for $E\subset \R^2$.

A closely related body of work studies analogous problems in subsets of vector spaces over finite fields.  Throughout, given a number $q$ which is a power of an odd prime, we let $\F_q$ denote the unique (up to isomorphism) field with $q$ elements.  The ``distance'' in the $d$-dimensional vector space $\F_q^d$ is defined algebraically, as follows.

\begin{dfn}
For $x\in\F_q^d$, define the \textbf{length} of $x$ by
\[
\|x\|=x_1^2+\cdots +x_d^2.
\]
Define the \textbf{distance} between $x,y\in \F_q^d$ to be the field element $\|x-y\|$.  For any $t\in\F_q$, define the \textbf{$t$-distance graph} $\cG_t(E)$ to be the graph with vertex set $E$ and adjacency relation given by $x\sim y$ if and only if $\|x-y\|=t$.  Finally, for any graph $G$, define the \textbf{$G$-distance set} $\De_G(E)\subset \F_q$ as in Definition \ref{distancegraphdefinition}, with respect to the finite field distance $\|\cdot\|$ in place of the Euclidean norm.
\end{dfn}

If $E\subset \F_q^d$, the analogue of dimension is the exponent $s$ such that $|E|=q^s$.  The analogue of the distance set $\De(E)$ or graph distance set $\De_G(E)$ having positive measure is to have the set contain a positive proportion of the possible elements of $\F_q$; that is, to have $|\De(E)|\approx q$.

Most of the results on point configurations in the Euclidean setting have corresponding results in the finite field setting, and the finite techniques are often closely related to the continuous ones.  Iosevich and Rudnev \cite{IR07} proved that there exists a constant $c$ such that if $|E|>cq^{\frac{d+1}{2}}$ then in fact $\De(E)=\F_q$; that is, not only does one get a positive proportion of all possible distances, but in fact every field element is realized as a distance between points of $E$.  The exponent was then improved \cite{CEHIK, MPPRS} in steps which mirrored the progress in the continuous setting.  Bennett, Chapman, Covert, Hart, Iosevich, and Pakianathan \cite{BCCHIP16} proved that there exist constants $c_n$ such that $|E|>c_n q^\frac{d+1}{2}$ implies $|\De_{P_k}(E)|\approx q$, where $P_n$ is the $n$-path graph ($C_n$ with an edge deleted), generalizing the result of \cite{IR07} to point configurations.  For cycles, Iosevich, Jardine, and B. McDonald \cite{IJM21} prove that for $d\geq 3$, there exist constants $c_n$ such that $|E|>c_nq^{\frac{d+2}{2}}$ implies $\De_{C_n}(E)=\F_q\setminus \{0\}$.  It is interesting to note that unlike in the continuous setting, these results does not depend on the parity of the cycle length.  In the case $d=2$, it can be deduced from results of Fitzpatrick, Iosevich, B. McDonald and Wyman \cite{FIMW} that $|E|>4q^{7/4}$ implies $\De_{C_4}(E)\supset \F_q\setminus \{0\}$, and $|E|>cq^{15/8}$ implies $\De_{C_6}(E)\supset \F_q\setminus \{0\}$.  Our second main result is as follows.

\begin{thm}
\label{FFmain}
Let $G$ be a graph satisfying
\[
\extr(G,n)\lesssim_{G,\al} n^{2-\al}
\]
for some $\al>0$.  There exists a constant $c$ (depending on $G$ and $\al$) such that if $E\subset \F_q^d$ is a set satisfying $|E|\geq c q^{\max(\frac{d+1}{2},\frac{1}{\al})}$, then $\De_G(E)\supset \F_q\setminus\{0\}$.
\end{thm}

In analogue with Corollaries \ref{aksADcorollary} and \ref{cyclesADcorollary}, we have the following.

\begin{cor}
\label{aksFFcorollary}
Let $E\subset \F_q^d$, let $G$ be a bipartite graph with parts $X$ and $Y$, and define
\[
r=\max_{v\in X}\deg(v).
\]
There exists a constant $c$ (depending on $G$ and $r$) such that if $|E|\geq c\cdot q^{\max(\frac{d+1}{2},r)}$, then $\De_G(E)\supset \F_q\setminus\{0\}$.
\end{cor}

\begin{cor}
\label{cyclesFFcorollary}
For each $k\geq 2$, there exists a constant $c_k$ such that the following holds.  Let $E\subset \F_q^d$ be such that $|E|>c_kq^\frac{d+1}{2}$.  Then:
\begin{enumerate}[(a)]
\item If $d=2$, then $\De_{C_{2k}}(E)\supset \F_q\setminus \{0\}$ if $k\geq 3$.
\item If $d\geq 3$, then $\De_{C_{2k}}(E)\supset \F_q\setminus \{0\}$ if $k\geq 2$.
\end{enumerate}
\end{cor}

We note that Corollary \ref{cyclesFFcorollary} improves the result of \cite{IJM21} on cycles in the finite field settings for even cycles in all dimensions $d\geq 3$, and improves the result of \cite{FIMW} for all even cycles except $C_4$ when $d=2$.

\subsection{Structure of the paper}
The idea of the proof is to study a finite graph which stands in for the distance graph $\cG_t(E)$.  In the finite field setting, this is simply the distance graph itself.  In the Ahlfors-David regular setting, it is a certain approximate distance graph based on a discretization of the set.  In each case, the goal is to obtain a lower bound on the number of edges of this finite graph under dimensional assumptions on the underlying set $E$.  The assumption $\extr(n,G)\lesssim n^{2-\al}$ then implies an upper bound on the number of edges in a graph avoiding copies of $G$.  If these bounds don't match, we must have a copy of $G$ in the (approximate) distance graph.

The paper is organized as follows.  Section \ref{regularizationsection} is devoted to understanding the structure of Ahlfors-David regular sets in a way that will ultimately allow us to define an approximate distance graph and estimate its number of edges.  This involves discretizing the set by covering by a finite collection of balls, and applying a regularization technique that allows us to say that many points are distance $t$ from many other points (in the language of graph theory, this can be interpreted as saying that many vertices have high degree).  In Section \ref{mainproofsection}, we prove our two main theorems (Theorem \ref{ADmain} and Theorem \ref{FFmain}).  Theorem \ref{FFmain} is more straightforward since the set in question is already finite, so we do this first.  To prove Theorem \ref{ADmain}, we use the same approach as the proof of Theorem \ref{FFmain} to show that the approximate distance graph has the desired subgraph.  The remainder of the section is used to prove that approximate copies of a graph can be turned into exact copies.  Finally, in Section \ref{applications}, we discuss several applications and prove Corollaries \ref{aksADcorollary}, \ref{cyclesADcorollary}, \ref{aksFFcorollary}, and \ref{cyclesFFcorollary}.

\section{Regularization and discretization of Ahlfors-David regular sets}
\label{regularizationsection}

If $\dim E>\frac{d+1}{2}$, then the Mattila-Sj\"{o}lin Theorem \cite{MS99} states that there exists an interval worth of $t>0$ such that $t\in\De(E)$.  However, $t\in\De(E)$ merely implies that there exist $x,y\in E$ such that $t=|x-y|$.  There are a number of refinements of this theorem which allow us to quantify the set of pairs $(x,y)$.  We follow an approach introduced by Iosevich and Taylor \cite{IT19}.  This lemma does not require Alhfors-David regularity, so we state a more general definition first.
\begin{dfn}
Let $E\subset \R^d$.  A \textbf{Frostman probability measure} on $E$ of exponent $s\in [0,d]$ is a probability measure $\mu$ which is supported on $E$ and satisfies
\[
\mu(B_r(x))\lesssim r^s
\]
for every $x\in\R^d,r>0$.
\end{dfn}
Note that, in particular, every $s$-regular measure is a Frostman probability measure of exponent $s$.  We denote by $\f$ a fixed smooth function supported on the unit ball, and let $\f(x)=\e^{-d}\f(x/\e)$ denote the corresponding approximate identity.  For any $t>0$, we denote by $\sig_t$ the normalized surface measure on the sphere of radius $t$ in $\R^d$, and for any $\e>0$ we write $\sig_t^\e=\sig_t*\f^\e$.  
\begin{lem}[Iosevich-Taylor \cite{IT19}, Lemma 2.1]
\label{ITlem}
Let $E\subset \R^d$, and let $\mu$ be a Frostman probability measure on $E$.  There exist constants $0<c<C,\de>0$ and an interval $I\subset (0,\infty)$ such that for any $t\in I$ and any sufficiently small $\e>0$, we have
\[
\mu(\{x\in E:c<\sig_t^\e*\mu(x)<C\})>\de,
\]
Moreover, the constants $c,C,\de$ do not depend on $\e$ or $t$.
\end{lem}

We will use a more geometric version of this lemma.  For $t,\e>0$ and $x\in \R^d$ we let $A_{t,e}(x)$ denote the annulus centered at $x$ of inner radius $t$ and thickness $\e$; that is,
\[
A_{t,\e}(x)=B_{t+\e}(x)\setminus B_t(x).
\]
\begin{lem}
\label{GeometricITlem}
Let $E\subset \R^d$, and let $\mu$ be a Frostman probability measure on $E$.  There exists a set $E'\subset E$ and an interval $I\subset (0,\infty)$ such that $\mu(E')\approx 1$, and for every $x\in E'$, $t\in I$, and sufficiently small $\e>0$ we have
\[
\mu(A_{t,\e}(x))\approx \e,
\]
with constant independent of $\e$ and $t$.
\end{lem}
\begin{proof}
Let $E,\mu$ be as in the statement of the lemma.  Let $c,C,\de$ be as given by Lemma \ref{ITlem}, and define
\[
E'=\{x\in E:c<\sig_t^\e*\mu(x)<C\}.
\]
By Lemma \ref{ITlem}, we have $\mu(E')>\de$ and $\sig_t^\e*\mu(x)\approx 1$ for every $x\in E'$.  Therefore,
\begin{align*}
1&\approx \int \sig_t^\e(x-y)d\mu(y) \\
&=\int\int \f^\e(x-y-z)d\sig_t(z)d\mu(y) \\
&\approx \e^{-d}\int\int_{|x-y-z|<\e}d\sig_t(z)d\mu(y).
\end{align*}
If $|x-y|-t>\e$ and $|z|=t$, then
\[
|x-y-z|\geq ||x-y|-t|>\e,
\]
hence the integral over the region $y\notin A_{t,\e}(x)$ is zero.  On the other hand, if $|x-y|<\frac{\e}{2}$, then
\[
\sig_t(\{z:|x-y-z|<\e\})\approx \e^{d-1}.
\]
Therefore,
\begin{align*}
1&\approx \e^{-d}\int\int_{|x-y-z|<\e}d\sig_t(z)d\mu(y) \\
&\approx \e^{-d}\e^{d-1}\int_{A_{t,\e}(x)}d\mu(y) \\
&\approx \e^{-1}\mu(A_{t,\e}(x)).
\end{align*}
\end{proof}

Our goal is to apply extremal graph theory results to the distance graphs $\cG_t(E)$, where $t$ is an arbitrary distance in the interval $I$ which is given by Lemmas \ref{ITlem}/\ref{GeometricITlem}.  In order to do this, we use the Alhfors-David assumption to discretize our set.

\begin{lem}[Discretization of AD-regular sets]
\label{discretization}
Let $E\subset \R^d$ be an Alhfors-David regular set with exponent $s$, and let $\mu$ be an $s$-regular measure on $E$.  For any $\e>0$, there exists a finite set $X=\{x_1,\dots,x_n\}\subset E$ such that the following hold:
\begin{enumerate}[(a)]
    \item $B_\e(x_i)\cap B_\e(x_j)=\emp $ for $i\neq j$,
    \item $E\subset \bigcup_{i=1}^n B_{3\e}(x_i)$,
    \item $n\approx \e^{-s}$.
\end{enumerate}
\end{lem}
\begin{proof}
We construct the points $x_i$ recursively.  First, let $x_1\in E$ be arbitrary.  Given $x_1,\dots,x_k$, if $E\subset \bigcup_{i=1}^k B_{3\e}(x_i)$ then we stop the construction.  Otherwise, let $x_{k+1}\in E\setminus \bigcup_{i=1}^k B_{3\e}(x_i)$ be arbitrary.  By construction, $|x_{k+1}-x_i|>3\e$ for all $i\leq k$, which proves (a).  If this construction continues through step $n$, then by additivity and the $s$-regularity of the measure $\mu$ we have
\[
1=\mu(E)\geq \sum_{i=1}^n \mu(B_{\e}(x_i))\approx n\e^s,
\]
so the process must terminate after $n\lesssim \e^{-s}$ steps.  Since the process only terminates when the tripled balls cover $E$, this proves (b) and the lower bound in (c).  Finally, by subadditivity and the $s$-regularity of $\mu$ we have
\[
1=\mu(E)\leq \sum_{i=1}^n \mu(B_{3\e}(x_i))\lesssim n\e^s,
\]
proving the upper bound in (c).
\end{proof}

\section{Proofs of Theorems \ref{ADmain} and \ref{FFmain}}
\label{mainproofsection}

For each of our main theorems (Theorems \ref{ADmain} and \ref{FFmain}), the proof idea is as follows.  If a set is big enough, then the distance graph (or, a discretized version of the distance graph) must have many edges.  Applying the hypothesis $\extr(n,G)\lesssim n^{2-\al}$, we conclude that there is a copy of $G$ in the distance graph.  In the finite field setting, the distance graph is already finite so no discretization is necessary, and a good edge count is already known.  We will give this proof first.  The starting point is a theorem of Iosevich and Rudnev which counts edges in the distance graph.

\begin{thm}[Iosevich-Rudnev \cite{IR07}]
\label{IRlemma}
Let $E\subset \F_q^d$ with $d\geq 2$.  For $t\in\F_q$, define
\[
\nu(t)=|\{(x,y)\in E\times E:\|x-y\|=t\}|.
\]
For any $t\neq 0$, we have
\[
\nu(t)=\frac{|E|^2}{q}+R(t),
\]
where the remainder term $R(t)$ satisfies the bound
\[
|R(t)|\leq 2q^{\frac{d-1}{2}}|E|.
\]
\end{thm}
An excellent exposition of Lemma \ref{IRlemma} can be found in \cite{CovertBook}.  We are now ready to prove our second main theorem.

\begin{proof}[Proof of Theorem \ref{FFmain}]
For any $t\neq 0$, the distance graph $\cG_t(E)$ has $n=|E|$ vertices and $e=\frac{1}{2}\nu(t)$ edges.  Therefore, if
\begin{equation}
\label{extremalcondition}
\nu(t)>2\extr(|E|,G)
\end{equation}
holds, we can conclude $\cG_t(E)$ contains an isomorphic copy of $G$, and hence $t\in \De_G(E)$.  By the hypothesis of the theorem, (\ref{extremalcondition}) is implied by
\begin{equation}
\label{reducedextremalcondition}
\nu(t)>2c_0|E|^{2-\al}
\end{equation}
for some constant $c_0$ (depending on $G$ and $\al$).  Let $c=\max((8c_0)^{1/\al},4)$, and suppose $|E|>cq^{\max(\frac{d+1}{2},\frac{1}{\al})}$.  This assumption implies that both
\begin{equation}
\label{IRcondition}
|E|>4q^{\frac{d+1}{2}}.
\end{equation}
and
\begin{equation}
\label{finitefieldexp}
|E|>(8c_0q)^{1/\al}.
\end{equation}
hold.  By Theorem \ref{IRlemma} and (\ref{IRcondition}) we have
\begin{equation}
\label{boundforbigsets}
\nu(t)>\frac{|E|^2}{q}-2q^{\frac{d-1}{2}}|E|>\frac{|E|^2}{4q}
\end{equation}
for all $t\neq 0$.  By (\ref{finitefieldexp}) we have
\[
2c_0|E|^{2-\al}<2c_0|E|^2(8c_0q)^{-1}=\frac{|E|^2}{4q},
\]
which together with (\ref{boundforbigsets}) implies (\ref{reducedextremalcondition}).
\end{proof}

In order to implement this idea in the setting of Ahlfors-David regular sets, we will use our discretization lemma (Lemma \ref{discretization}) to obtain a finite set of points in $E$ which we can use to represent the set.  We use these points to define an approximate version of the distance graph, where two points are adjacent if and only if their distance is close to $t$.  We then apply our regularization lemma (Lemma \ref{GeometricITlem}) to conclude that this graph has many vertices with large degree, which implies that it has many edges.  From this, we conclude that there exists an approximate copy of $G$.  We would like to then pass to limits and say that these approximate copies of $G$ converge to an exact copy of $G$.  However, there is an important subtlety in this argument.  In order to have a true isomorphic copy of $G$, it is necessary to ensure that distinct vertices in our approximate copy of $G$ do not converge to a common limit.  For example, we must avoid the situation shown in Figure \ref{degeneracy}, where an appoximate copy of $C_4$ converges to an exact copy of $P_2$.

\begin{figure}
\centering
\begin{minipage}{.5\textwidth}
  \centering
\begin{tikzpicture}[scale=2, thick]
  \draw (0,0)--(1.1,0.9)--(2,2)--(0.9,1.1)--(0,0);
   \draw [fill, gray, opacity=.3] (0,0) circle [radius=0.1];
   \draw [fill, gray, opacity=.3] (1.1,0.9) circle [radius=0.1];
   \draw [fill, gray, opacity=.3] (.9,1.1) circle [radius=0.1];
   \draw [fill, gray, opacity=.3] (2,2) circle [radius=0.1];

   \node[below left] at (0,0) {$x_1^\e$};
   \node[below right] at (1.1,.9) {$x_2^\e$};
   \node[above right] at (2,2) {$x_3^\e$};
   \node[above left] at (.9,1.1) {$x_4^\e$};
\end{tikzpicture}
\end{minipage}%
\begin{minipage}{.5\textwidth}
  \centering
\begin{tikzpicture}[scale=2, thick]
 \draw (0,0)--(1,1)--(2,2);
 \draw [fill] (0,0) circle [radius=0.05];
   \draw [fill] (1,1) circle [radius=0.05];
   \draw [fill] (2,2) circle [radius=0.05];
 \node[below left] at (0,0) {$x_1$};
   \node[below right] at (1,1) {$x_2=x_4$};
   \node[above right] at (2,2) {$x_3$};
\end{tikzpicture}
\end{minipage}
\caption{The approximate 4-cycle $(x_1^\e,x_2^\e,x_3^\e,x_4^\e)$ converging to a degenerate configuration $(x_1,x_2,x_3,x_2)$.}
  \label{degeneracy}
\end{figure}
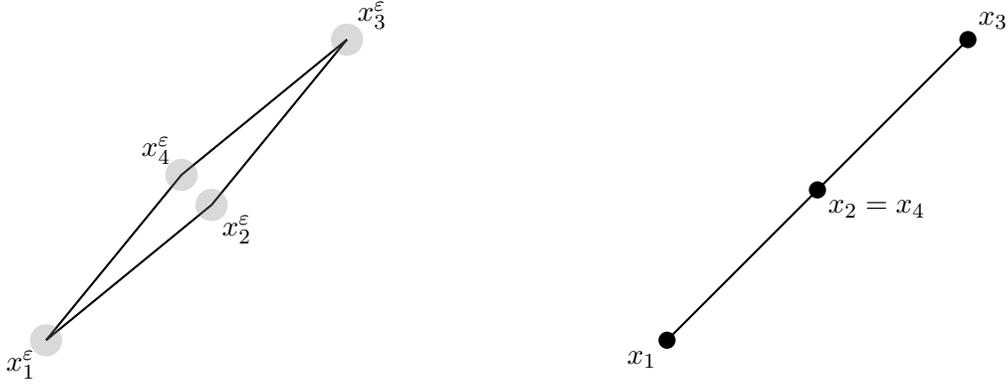

In order to handle this issue, we start with making a precise definition about what kinds of approximations we will allow.

\begin{dfn}
Let $\e,t>0$, and let $G$ be a graph with vertex set $\{1,\dots,m\}$ and adjacency relation $\sim$.  We say that a configuration
\[
x=(x_1,\dots,x_n)\in (\R^d)^m
\]
is a \textbf{$(G,t,\e)$-approximation} if the following conditions hold:
\begin{enumerate}[(a)]
    \item $||x_i-x_j|-t|<\e$ whenever $i\sim j$
    \item $|x_i-x_j|>3\e$ for each $i\neq j$.
\end{enumerate}
\end{dfn}

The standard technique for proving that configurations are not all degenerate is to define a measure on the space of configurations, and prove that the set of degenerate configurations has measure zero.  This is accomplished with the following lemma.

\begin{lem}
\label{weakstar}
Let $K\subset \R^d$ be a compact set, and for each $\e>0$, let $\La_\e$ be a Borel probability measure supported on $K$.  Then, there exists a sequence $\e_j\to 0$ such that for each Borel set $A\subset K$, the limit
\[
\La(A):=\lim_{j\to\infty} \La_{\e_j}(A)
\]
exists and defines a probability measure supported on $K$.
\end{lem}
\begin{proof}
Let $\cM(K)$ denote the space of finite complex Radon measures supported on $K$.  This is a Banach space under the total variation norm, and is dual to the (separable) space of continuous functions on $K$.  By the sequential version of the Banach-Alaoglu theorem, the closed unit ball is sequentially compact in the weak-$*$ topology.  Since any probability measure has norm 1, the set $\{\La_\e:\e>0\}$ has a weak-$*$ limit point $\La\in\cM(K)$.  Letting $\La_{\e_j}$ be a sequence with $\La$ as the weak-$*$ limit, we have
\[
\La(A)=\lim_{j\to\infty} \La_{\e_j}(A)
\]
for any Borel $A\subset K$.  In particular, 
\[
\La(K)=\lim_{j\to\infty} \La_{\e_j}(K)=1,
\]
hence $\La$ is a probability measure supported on $K$.
\end{proof}

\begin{lem}[Approximate distance graphs imply exact distance graphs]
\label{approximatetoexact}
Let $\mu$ be a probability measure on $\R^d$, let $E=\textup{supp } \mu$, let $t>0$, and let $G$ be a finite graph with vertex set $\{1,\dots,m\}$ and adjacency relation $\sim$.  Suppose that there exists $\e_0>0$ such that for every $\e\in (0,\e_0)$, the set $E^m$ contains a $(G,t,\e)$-configuration.  Then, $t\in\De_G(E)$.

\end{lem}

\begin{proof} For each $\e>0$, let
\[
(x_1^\e,\dots,x_m^\e)\in E^m
\]
denote a $(G,t,\e)$-approximation.  Since $x_i^\e\in \textup{supp } \mu$ for each $i\in\{1,\dots,m\}$, we have $\mu(B_\e(x_i^\e))>0$ for any $\e>0$.  Let $\La_\e$ denote the normalized restriction of $\mu^{n}$ to the set
\[
B_\e(x_1^\e)\times\cdots\times B_\e(x_{m}^\e)\subset (\R^{d})^{m}.
\]
Apply Lemma \ref{weakstar} with these measures to get a weak-$*$ limit probability measure $\La$.  If $(x_1,\dots,x_{m})$ is in the support of $\La$, then it must be in the support of $\La_\e$ for arbitrarily small values of $\e$, hence $x_i$ must be a limit point of $\{x_i^\e\}_\e$.  It follows immediately that $x_i\in E$ (since $E$ is closed), and $|x_i-x_j|=t$ for every $i\sim j$.  It also follows that $x_i\neq x_j$ if $i\neq j$, since by definition of a $(G,t,\e)$-cycle, the balls $B_\e(x_i^\e)$ and $B_\e(x_j^\e)$ are disjoint.  Therefore, $\Lambda$ is a probability measure supported on the set
\[
\{(x_1,\dots,x_m)\in E^{m}:x_i \text{ distinct, } |x_i-x_j|=t \text{ whenever } i\sim j\}.
\]
Any configuration in this set must consist of the vertices of an isomorphic copy of the graph $G$ in the distance graph $\cG_t(E)$, so $t\in \De_G(E)$.
\end{proof}


We are now ready to prove Theorem \ref{ADmain}.

\begin{proof}[Proof of Theorem \ref{ADmain}]
If $E$ is an Ahlfors-David regular set of exponent $s$ and $\mu$ is an $s$-regular measure, then $\textup{supp }\mu=E$.  Therefore, by Lemma \ref{approximatetoexact} it suffices to prove that for every sufficiently small $\e>0$ and every $t$ in some interval, the set $E$ contains the vertices of an $(G,t,\e)$-approximation.  Let $E'$ and $I$ be as in Lemma \ref{GeometricITlem}, and for the remainder of the proof let $t\in I$ be fixed.  Let $X$ be as in Lemma \ref{discretization}, and define a graph $H$ with vertex set $X$ and adjacency defined by $x \sim_H y$ if and only if $||x-y|-t|<10\e$.  Then, $H$ has $n\approx \e^{-s}$ vertices, and it remains to estimate the number of edge.  Define
\[
X'=\{x\in X: E'\cap B_{3\e}(x)\neq\emp\}.
\]
We have
\[
1\approx \mu(E')\leq \sum_{x\in X'}\mu(B_{3\e}(x))\lesssim |X'|\e^s,
\]
hence $|X'|\approx \e^{-s}$.  For any $x,y\in X$, we have $x\sim_H y$ whenever $B_{3\e}(y)\cap A_{t,\e}(x)\neq\emp$.  Therefore, for $x\in X'$,
\[
\e\approx \mu(A_{t,\e}(x))\leq \sum_{\substack{y\in X \\ y\sim x}} \mu(B_{3\e}(y))\lesssim \e^{s}\deg(x),
\]
hence $\deg(x)\gtrsim \e^{1-s}$ for $x\in X'$.  Putting these bounds together, if $e$ is the number of edges in $G$, then
\[
e\approx \sum_{x\in X}\deg(x)\gtrsim |X'|\e^{1-s}\approx \e^{1-2s}.
\]
If $s>\frac{1}{\al}$, then for any $C>0$ there exists $\e_0$ such that
\[
\e^{1-2s}>C(\e^{-s})^{2-\al}
\]
for all $\e\in (0,\e_0)$.  If we take $C$ to be the implicit constant in the hypothesis $\extr(G,n)\lesssim n^{2-\al}$, then we conclude that $H$ contains an isomorphic copy of $G$.  By construction, the vertices of this copy of $G$ form a $(G,t,10\e)$-approximation in $E$.  By Lemma \ref{approximatetoexact}, it follows that $I\subset \De_G(E)$.
\end{proof}

\section{Examples and applications}
\label{applications}
\subsection{Proof of Corollaries \ref{aksADcorollary}, \ref{cyclesADcorollary}, \ref{aksFFcorollary}, and \ref{cyclesFFcorollary}.}
Our corollaries are based on the following theorem, which appears as Corollary 2.3 in \cite{AKS03}.
\begin{thm}[Alon-Krivelevich-Sudakov \cite{AKS03}]
\label{AKStheorem}
Let $G$ be a bipartite graph with parts $X$ and $Y$, and let
\[
r=\max_{v\in X}\deg(v).
\]
Then,
\[
\extr(n,G)\lesssim n^{2-\frac{1}{r}}
\]
\end{thm}
\begin{proof}[Proof of Corollaries \ref{aksADcorollary} and \ref{aksFFcorollary}]
If $G$ satisfies the hypothesis of Corollary \ref{aksADcorollary} (respectively, Corollary \ref{aksFFcorollary}), then by Theorem \ref{AKStheorem}, it also satisfies the hypothesis of Theorem \ref{ADmain} (respectively, Theorem \ref{FFmain}) with $\al=\frac{1}{r}$.  The conclusion follows from those theorems.  
\end{proof}

\begin{proof}[Proof of Corollaries \ref{cyclesADcorollary}(b) and \ref{cyclesFFcorollary}(b)]
For any $k\geq 2$, the cycle $C_{2k}$ is a bipartite graph in which every vertex has degree $2$.  If $d\geq 3$, then
\[
\max\left(\frac{d+1}{2},2\right)=\frac{d+1}{2},
\]
so the assumption $s>\frac{d+1}{2}$ allows us to apply Corollary \ref{aksADcorollary} with $G=C_{2k}$ and $r=2$ to get the conclusion of Corollary \ref{cyclesADcorollary}(b).  Using Corollary \ref{aksFFcorollary} instead, we get the conclusion of Corollary \ref{cyclesFFcorollary}(b).
\end{proof}
Note that the above proof does not work in the case $d=2$, as in this case we have
\[
\max\left(\frac{d+1}{2},2\right)=2.
\]
This does not yield a non-trivial exponent.  To get a result on cycles in the plane, we need a better bound than Theorem \ref{AKStheorem}.  We may use the following.
\begin{thm}[Bondy-Simonovits \cite{BS74}]
\label{BStheorem}
For any $k\geq 2$, we have
\[
\extr(n,C_{2k})\lesssim_k n^{1+\frac{1}{k}}.
\]
\end{thm}

\begin{proof}[Proof of Corollaries \ref{cyclesADcorollary}(a) and \ref{cyclesFFcorollary}(a)]
Theorem \ref{BStheorem} implies that $\extr(n,C_{2k})\lesssim n^{2-\al}$ with
\[
\al=1-\frac{1}{k}.
\]
If $k\geq 3$ then $\al\geq  \frac{2}{3}$, hence
\[
\max\left(\frac{3}{2},\frac{1}{\al}\right)=\frac{3}{2}.
\]
We can therefore apply Theorem \ref{ADmain} to get Corollary \ref{cyclesADcorollary}(a).  By applying Theorem \ref{FFmain} instead, we get Corollary \ref{cyclesFFcorollary}.
\end{proof}

\subsection{The hypercube graph $Q_k$}
\begin{dfn}
The \textbf{hypercube graph}, denoted $Q_k$, is the graph with vertex set $\{0,1\}^k$ and adjacency relation defined by $
(a_1,\dots,a_k)\sim (b_1\dots,b_k)$ if and only if $a_i\neq b_i$ for exactly one $i\in\{1,\dots,k\}$.
\end{dfn}
\begin{figure}
\centering
\begin{minipage}{.5\textwidth}
  \centering
\begin{tikzpicture}[scale=2, thick]
  \draw (0,0)--(1,0)--(1.4,.3)--(.4,.3)--(0,0);
  \draw [fill] (0,0) circle [radius=0.05];
  \draw [fill] (1,0) circle [radius=0.05];
  \draw [fill] (1.4,.3) circle [radius=0.05];
  \draw [fill] (.4,.3) circle [radius=0.05];

  \node[below left] at (0,0) {$00$};
  \node[below left] at (1,0) {$10$};
  \node[above right] at (.4, .3) {$01$};
  \node[above right] at (1.4,.3) {$11$};
\end{tikzpicture}
\end{minipage}%
\begin{minipage}{.5\textwidth}
  \centering
\begin{tikzpicture}[scale=2, thick]
 \draw (0,0)--(1,0)--(1.4,.3)--(.4,.3)--(0,0);
 \draw (0,.8)--(1,.8)--(1.4,1.1)--(.4,1.1)--(0,.8);
 \draw (0,0)--(0,.8);
 \draw (1,0)--(1,.8);
 \draw (1.4,.3)--(1.4,1.1);
 \draw (.4,.3)--(.4,1.1);
  \draw [fill] (0,0) circle [radius=0.05];
  \draw [fill] (1,0) circle [radius=0.05];
  \draw [fill] (1.4,.3) circle [radius=0.05];
  \draw [fill] (.4,.3) circle [radius=0.05];
  \draw [fill] (0,.8) circle [radius=0.05];
  \draw [fill] (1,.8) circle [radius=0.05];
  \draw [fill] (1.4,1.1) circle [radius=0.05];
  \draw [fill] (.4,1.1) circle [radius=0.05];

  \node[below left] at (0,0) {$000$};
  \node[below left] at (1,0) {$100$};
  \node[above right] at (.4, .3) {$010$};
  \node[above right] at (1.4,.3) {$110$};
  \node[below left] at (0,0+.8) {$001$};
  \node[below left] at (1,0+.8) {$101$};
  \node[above right] at (.4, .3+.8) {$011$};
  \node[above right] at (1.4,.3+.8) {$111$};
\end{tikzpicture}
\end{minipage}
\caption{The graphs $Q_2$ and $Q_3$.}
  \label{hypercube}
\end{figure}
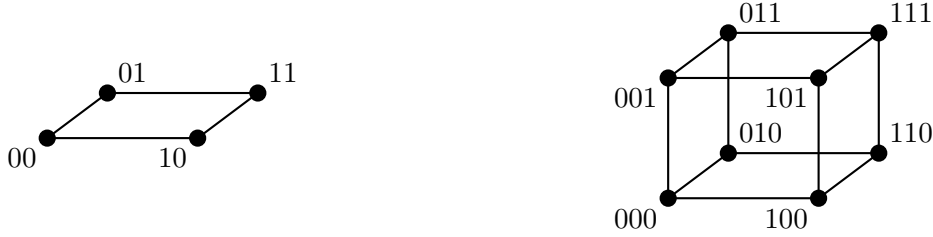
Geometrically, $Q_k$ can be understood as the vertices of a $k$-dimensional hypercube, with edges corresponding to the edges of the cube (Figure \ref{hypercube}).  Note that $Q_k$ is a bipartite graph, with independent sets given by sorting vertices $(a_1,\dots,a_k)\in\{0,1\}^k$ according to whether $\sum_i a_i$ is even or odd.  Further, every vertex has degree $k$.  We can therefore apply Corollary \ref{aksADcorollary} to conclude that for $k<d$, a set of dimension $s>k$ in either the Ahlfors-David regular or finite field setting has an abundance of hypercubes.  When $k=1,2$ this does not say anything new, since $Q_1=K_2$ and $Q_2=C_4$.  For $k\geq 3$ we can do better by applying a classical result of Erd\"{o}s and Simonovits \cite{ES69} when $k=3$, and a recent result of Janzer and Sudakov \cite{JS24} when $k\geq 4$.

\begin{thm}[Erd\"{o}s-Simonovits \cite{ES69}]
\label{EStheorem}
We have the bound 
\[
\extr(n,Q_3)\lesssim n^{8/5}.
\]
\end{thm}
\begin{thm}[Janzer-Sudakov \cite{JS24}]
\label{JStheorem}
For any $k\geq 4$, we have
\[
\extr(n,Q_k)\lesssim n^{2-\frac{1}{k-1}+\frac{1}{(k-1)2^{k-1}}}.
\]
\end{thm}
\begin{cor}
\label{hypercubeADcorollary}
Let $d\geq 3$, and let $E\subset \R^d$ be an Ahlfors-David regular set of exponent $s$.  
\begin{enumerate}[(a)]
\item If $s>\max(\frac{5}{2},\frac{d+1}{2})$, then $\De_{Q_3}(E)$ has non-empty interior.
\item If $k\geq 4$ and
\[
s>\max\left(\frac{2^{k-1}}{2^{k-1}-1}(k-1),\frac{d+1}{2}\right),
\]
then $\De_{Q_k}(E)$ has non-empty interior.
\end{enumerate}

\end{cor}
\begin{proof}
To prove (a), Theorem \ref{EStheorem} is equivalent to the bound
\[
\extr(n,Q_3)\lesssim n^{2-\frac{2}{5}}.
\]
Applying Theorem \ref{ADmain} with $\al=\frac{2}{5}$, we get the result.  To prove (b), we use Theorem \ref{JStheorem} in place of Theorem \ref{EStheorem} to get $\extr(n,Q_k)\lesssim n^{2-\al}$ with
\[
\al=\frac{2^{k-1}-1}{2^{k-1}(k-1)}.
\]
\end{proof}

\begin{cor}
Let $d\geq 3$.  For each $k\geq 3$, there is a constant $c_k$ such that for any $E\subset\F_q^d$, the following hold.
\begin{enumerate}[(a)]
\item If $|E|>c_3q^{\max(\frac{5}{2},\frac{d+1}{2})}$, then $\De_{Q_3}(E)\supset \F_q\setminus\{0\}$.
\item If $k\geq 4$ and $|E|>c_kq^s$, where
\[
s=\max\left(\frac{2^{k-1}}{2^{k-1}-1}(k-1),\frac{d+1}{2}\right),
\]
then $\De_{Q_k}(E)\supset \F_q\setminus\{0\}$.
\end{enumerate}
\end{cor}
\begin{proof}
The proof is the same as the proof of Corollary \ref{hypercubeADcorollary}, except we use Theorem \ref{FFmain} in place of Theorem \ref{ADmain}.
\end{proof}

\subsection{Connections to VC-dimension}
VC-dimension is a notion arising in statistical learning theory, where it is used to understand the amount of training data needed to accurately solve a learning problem.  Recently, it has been studied in connection to point configuration problems in both the finite field setting \cite{Small1,FIMW,IMS23,MSW25} and the continuous setting \cite{IMMM25}.  The definition is as follows.

\begin{dfn}
\label{VCdefinition}
Let $X$ be a set, and let $\cH$ be a collection of functions $X\to \{0,1\}$.  We say that $\cH$ \textbf{shatters} a finite set $F=\{x_1,\dots,x_k\}\subset X$ if for each set of indices $I\subset\{1,\dots,k\}$, there exists a function $h_I\in\cH$ satisfying
\[
h_I(x_i)=\chi_I(i),
\]
where $\chi_I$ denotes the characteristic function of the set $I$.  The \textbf{VC-dimension} of $\cH$ is the largest number $k$ such that $\cH$ shatters some set of $k$ points (if no such number exists, the VC-dimension is said to be infinite).
\end{dfn}
The relation $h_I(x_i)=\chi_I(i)$ in Definition \ref{VCdefinition} can be naturally phrased in the language of graph theory.
\begin{dfn}
For each $k\in\N$, define the \textbf{$k$-shattering graph} $S_k$ to be the bipartite graph with parts $\{1,\dots,k\}$ and $\cP(\{1,\dots,k\})$ (here, $\cP(\cdot)$ denotes the power set), with adjacency given by the element relation (i.e., $i\sim I$ if and only if $i\in I$).
\end{dfn}

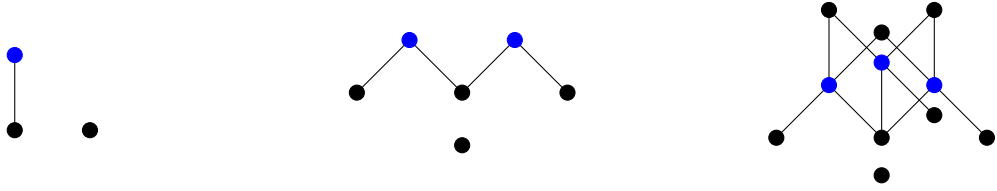
\begin{figure}
\centering
\begin{minipage}{.33\textwidth}
  \centering
\begin{tikzpicture}
\draw(0,0)--(0,1);
\draw[fill](0,0)circle[radius=0.1];
\draw[fill, blue](0,1)circle[radius=0.1];
\draw[fill](1,0)circle[radius=0.1];
\end{tikzpicture} 
\end{minipage}%
\begin{minipage}{.33\textwidth}
  \centering
\begin{tikzpicture}
\draw(-1.4,0)--(-.7,.7)--(0,0)--(.7,.7)--(1.4,0);
\draw[fill](0,0)circle[radius=0.1];
\draw[fill](-1.4,0)circle[radius=0.1];
\draw[fill](1.4,0)circle[radius=0.1];
\draw[fill, blue](.7,.7)circle[radius=0.1];
\draw[fill, blue](-.7,.7)circle[radius=0.1];
\draw[fill](0,-.7)circle[radius=0.1];
\end{tikzpicture}
\end{minipage}
\begin{minipage}{.33\textwidth}
  \centering
\begin{tikzpicture}
\draw(0,0)--(.7,.7);
\draw(0,0)--(0,1);
\draw(0,0)--(-.7,.7);
\draw(.7,.7)--(.7,1.7)--(0,1);
\draw(.7,.7)--(0,1.4)--(-.7,.7);
\draw(-.7,.7)--(-.7,1.7)--(0,1);
\draw(-.7,.7)--(-1.4,0);
\draw(.7,.7)--(1.4,0);
\draw(0,1)--(.7,.3);

\draw[fill](0,0)circle[radius=0.1];

\draw[fill](.7,1.7)circle[radius=0.1];
\draw[fill](-.7,1.7)circle[radius=0.1];
\draw[fill](0,1.4)circle[radius=0.1];

\draw[fill](-1.4,0)circle[radius=0.1];
\draw[fill](1.4,0)circle[radius=0.1];
\draw[fill](.7,.3)circle[radius=0.1];

\draw[fill, blue](.7,.7)circle[radius=0.1];
\draw[fill, blue](-.7,.7)circle[radius=0.1];
\draw[fill, blue](0,1)circle[radius=0.1];

\draw[fill](0,-.5)circle[radius=0.1];
\end{tikzpicture}
\end{minipage}
\caption{The graphs $S_1,S_2,S_3$.}
  \label{shattering}
\end{figure}

The shattering graphs for $k=1,2,3$ are shown in Figure \ref{shattering} (the vertices $\{1,\dots,k\}$ are shown in blue, and the sets $I\subset \{1,\dots,k\}$ are shown in black).  We consider the following family of indicator functions in our setting.
\begin{dfn}
Given $E\subset \R^d$ (respectively, $E\subset \F_q^d$) and $t>0$ (respectively, $t\in\F_q\setminus\{0\}$), let $\cH_t(E)$ denote the set of indicator functions of spheres of radius $t$ centered at points of $E$.
\end{dfn}
It is not hard to show that the VC-dimension of $\cH_t(\R^d)$ or $\cH_t(\F_q^d)$ is $d+1$ for any $t$.  Therefore, a natural question to ask is whether $\cH_t(E)$ must have VC-dimension $d+1$ provided $E$ is sufficiently large in the appropriate sense.  The problem was introduced in the finite field setting by Fitzpatrick, Iosevich, B. McDonald and Wyman \cite{FIMW}, who showed that if $E\subset \F_q^2$ is a set of size $|E|>cq^{15/8}$, then the VC-dimension of $\cH_t(E)$ is 3 for all $t\neq 0$.  In the continuous setting, the problem was studied in joint work by the author and Magyar, Iosevich, and B. McDonald \cite{IMMM25}, who show that when $d\geq 3$, there is a dimensional threshold $s<d$ which guarantees that the VC-dimension of $\cH_t(E)$ is at least 3 for an interval of $t$.  

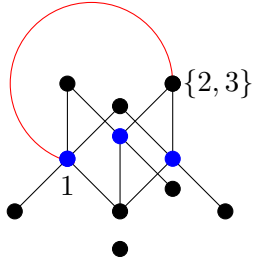
\begin{figure}
\centering
\begin{tikzpicture}
\draw(0,0)--(.7,.7);
\draw(0,0)--(0,1);
\draw(0,0)--(-.7,.7);
\draw(.7,.7)--(.7,1.7)--(0,1);
\draw(.7,.7)--(0,1.4)--(-.7,.7);
\draw(-.7,.7)--(-.7,1.7)--(0,1);
\draw(-.7,.7)--(-1.4,0);
\draw(.7,.7)--(1.4,0);
\draw(0,1)--(.7,.3);

\draw[fill](0,0)circle[radius=0.1];

\draw[fill](.7,1.7)circle[radius=0.1];
\draw[fill](-.7,1.7)circle[radius=0.1];
\draw[fill](0,1.4)circle[radius=0.1];

\draw[fill](-1.4,0)circle[radius=0.1];
\draw[fill](1.4,0)circle[radius=0.1];
\draw[fill](.7,.3)circle[radius=0.1];

\draw[fill, blue](.7,.7)circle[radius=0.1];
\draw[fill, blue](-.7,.7)circle[radius=0.1];
\draw[fill, blue](0,1)circle[radius=0.1];

\draw[fill](0,-.5)circle[radius=0.1];

\draw[red] (.7,1.7) arc [start angle=00, end angle=258, radius=1.08] ;
\draw[fill](.7,1.7)circle[radius=0.1];
\draw[fill, blue](-.7,.7)circle[radius=0.1];

\node[below] at (-.7,.6) {$1$};
\node[right] at (.7,1.7) {$\{2,3\}$};

\end{tikzpicture}
\caption{$S_3$ with an extra edge (red).}
  \label{badshattering}
\end{figure}

The results \cite{FIMW, IMMM25} are both approached by viewing the problem as a point configuration problem, where one must find a copy of $S_k$ in the distance graph.  However, the problem differs from other point configuration problems we have considered, in that one needs to have $S_k$ as not just a subgraph but an \textit{induced} subgraph.  If the distance graph contains a copy of $S_3$ (for instance), but the points representing the index $1$ and the set of indices $\{2,3\}$ happen to have the given distance $t$, then the copy of $S_3$ in the distance graph has an extra edge (Figure \ref{badshattering}).  In Figure \ref{badshattering}, the black points do not really shatter the blue points; instead, we have represented the function $(1,2,3)\mapsto (1,1,1)$ twice, and we have not represented the function $(1,2,3)\mapsto (0,1,1)$ at all.  Since the techniques in our present paper are not designed to handle this problem, we do not get new results on VC-dimension.  However, we can apply our methods to show that the graph $S_k$ appears in the distance graph.

\begin{cor}
\label{VCADcor}
Let $d\geq 3$ and let $k<d$.  Let $E\subset \R^d$ be a compact, Ahlfors-David regular set of exponent $s>\max(\frac{d+1}{2},k)$.  Then, $\De_{S_k}(E)$ has non-empty interior.
\end{cor}
\begin{proof}
The graph $S_k$ is bipartite, with one of the independent sets having size $k$.  Therefore, every vertex in the other independent set has degree $k$.  By Corollary \ref{aksADcorollary}, the result follows.
\end{proof}

\begin{cor}
Let $d\geq 3$.  For each $k$, there exists a constant $c_k$ with the following property.  Let $k<d$, and let $E\subset \F_q^d$ be a set satisfying $|E|>c_kq^{\max(\frac{d+1}{2},k)}$.  Then, $\De_{S_k}(E)\supset \F_q\setminus \{0\}$.
\end{cor}
\begin{proof}
The proof is the same as the proof of Corollary \ref{VCADcor}, except we use Corollary \ref{aksFFcorollary} in place of Corollary \ref{aksADcorollary}.
\end{proof}

\bibliographystyle{plain}
\bibliography{refsCycles}

\end{document}